\documentclass[12pt,a4paper]{amsart}

\RequirePackage[ansinew]{inputenc}
\usepackage[english]{babel}
\usepackage{amsmath,amsfonts,amsthm,amssymb}\usepackage[top=3.5cm,left=3cm,right=3cm,bottom=3.5cm]{geometry}
\usepackage{fourier}
\usepackage[usenames,dvipsnames]{color}
\usepackage[colorlinks=true,linkcolor=blue,citecolor=blue]{hyperref}

\usepackage{tikz}
\usepackage{multicol}

\newcommand{\C}{\mathbb{C}}
\newcommand{\R}{\mathbb{R}}

\newcommand{\del}{\partial}

\newcommand{\ddc}{\text{dd}^c}
\newcommand{\diff}{\text{\normalfont d}}

\newcommand{\vol}{\text{\normalfont vol}}

\newcommand{\D}{\mathbb{D}}

\newcommand{\pr}{\mathbb{P}}

\newtheorem{theorem}{Theorem}[section]
\newtheorem{proposition}[theorem]{Proposition}
\newtheorem{corollary}[theorem]{Corollary}
\newtheorem{lemma}[theorem]{Lemma}

\newtheorem*{theorem*}{Theorem}

\theoremstyle{definition}
\newtheorem{definition}[theorem]{Definition}
\newtheorem{example}[theorem]{Example}

\newtheorem{remark}[theorem]{Remark}

\numberwithin{equation}{section}

\title{Self-intersection of Foliated Cycles on Complex Manifolds}

\author{Lucas Kaufmann}
\address{Sorbonne Universités, UPMC Univ Paris 06, IMJ-PRG, UMR 7586 CNRS, Univ Paris Diderot, Sorbonne Paris Cité, F-75005, Paris, France}
\email{lucas.kaufmann@imj-prg.fr}
\date{}

\begin{document}
\maketitle

\begin{abstract}
Let $X$ be a compact Kähler manifold and let $T$ be a foliated cycle directed by a transversally Lipschitz lamination on $X$. We prove that the self-intersection of the cohomology class of $T$ vanishes as long as $T$ does not contain currents of integration along compact manifolds. 

As a consequence we prove that  Lipschitz laminations of low codimension in certain manifolds, e.g.\ projective spaces, do not carry any foliated cycles except those given by integration along compact leaves. 
\end{abstract}

\section{Introduction}

Let $X$ be a compact complex manifold. An embedded lamination $\mathcal L$ of dimension $q$ in $X$ is a compact subset of $X$ given locally by a disjoint union of holomorphic graphs over a $q$-dimensional polydisc. These local graphs are called plaques of $\mathcal L$ (see section \ref{sec:laminations} for a precise definition). When $q=1$ we say that $\mathcal L$ is a \textit{Riemann surface lamination} on $X$. These  objects have been considered by numerous authors and they have connections to many other branches of Geometry and Dynamics such as complex dynamical systems of several variables and foliation theory.  For an account of the subject the reader may consult the surveys \cite{fornaess-sibony:laminations}, \cite{ghys:laminations-survey}  and the reference therein.

Given an embedded lamination $\mathcal L$ of dimension $q$, a foliated cycle $T$ directed by $\mathcal L$ is a positive closed current on $X$ that is locally given by an average of currents of integration along the plaques with respect to a positive measure $\mu$. We say that $T$ is diffuse if the measure $\mu$ has no atoms. This is equivalent to the fact that $T$ has no mass on submanifolds of dimension $q$ (see Remark \ref{rmk:atoms-leaves}). Alternatively, a foliated cycle may be described as a holonomy invariant transverse measure.

From a geometric point of view it is expected that the self-intersection of a diffuse foliated cycle should be zero. For the homological intersection this was proved by Hurder--Mitsumatsu,  \cite{hurder-mitsumatsu} in the context of (real) foliations of arbitrary dimension and codimension. For complex laminations, the main results concern one-dimensional cycles on complex surfaces. 

The vanishing of the self-intersection indicates that the existence of a foliated cycle is quite special. As a matter of fact, many laminations  carry no foliated cycle at all. For laminations in $\pr^2$ this was first noted by Camacho--Lins-Neto--Sad in \cite{camacho-lins-neto-sad}, where they considered cycles supported by minimal sets of holomorphic foliations. For transversally smooth foliations the result follows from Hurder--Mitsumatsu's Theorem and for Lipschitz laminations this is due to Forn\ae ss and Sibony (see  \cite{fornaess-sibony:finite-energy} and \cite{fornaess-sibony:laminations}).

Our main result concerns the self--intersection of diffuse foliated cycles of arbitrary dimension and codimension on a compact Kähler manifold.

\begin{theorem} \label{thm:main-thm}
Let $X$ be a compact Kähler manifold of dimension $n$ and let $T$ be a diffuse foliated cycle of dimension $q$ directed by an embedded transversally Lipschitz lamination on $X$. If $\{T\} \in H^{n-q,n-q}(X)$ denotes its cohomology class then $\{T\} \smallsmile \{T\} = 0$ in $H^{2n-2q,2n-2q}(X)$.
\end{theorem}

The proof uses a notion of intersection of currents introduced by Dinh and Sibony in \cite{dinh-sibony:density}, which is denoted by the symbol $\curlywedge$. Our arguments also  show that $T \curlywedge T = 0$ if $q \geq \frac{n}{2}$, see Section \ref{sec:global-result}. This is related to a theorem of Dujardin \cite{dujardin:intersection} saying that for a foliated cycle $T$ defined on an open set of $\C^2$, the wedge product $T \wedge T$ vanishes as long as it is well-defined in the usual sense of Pluripotential Theory. Our theorem extends this result in two ways, firstly it works in any dimension and codimension (as long as $q \geq \frac{n}{2}$) and secondly, no a priori regularity of $T$ is assumed. Nevertheless we need the underlying lamination to be transversally Lipschitz and the current $T$ to be globally defined.

As an application of Theorem \ref{thm:main-thm} we can show that laminations of low codimension on certain Kähler manifolds  carry no diffuse foliated cycles (see section \ref{sec:applications} for other examples).

\begin{corollary} \label{cor:main-cor}
Let $\mathcal L$ be an embedded transversally Lipschitz lamination of dimension $q \geq n/2$ on $\pr^n$. Then there is no diffuse foliated cycle directed by $\mathcal L$. In particular, if  $\mathcal L$ has no compact leaves then there is no foliated cycle directed by $\mathcal L$.
\end{corollary}


The presence of a foliated cycles imposes strong conditions on the dynamics of the lamination and as noted above their existence is quite rare. Nevertheless, every lamination  by Riemann surfaces carries a  \textit{directed harmonic current}, this is a positive $\ddc-$ closed current directed by the leaves of $\mathcal L$ (see \cite{fornaess-sibony:laminations} for more details). This was proven in \cite{garnett} for foliations, in \cite{berndtsson-sibony} for continuous laminations with singularities and in \cite{sibony:pfaff} in greater generality. However, there are examples of laminations of dimension two that do not carry any directed harmonic current, see \cite{fornaess-sibony-wold}. We should also notice that the non-existence of foliated cycles for Riemann surface laminations can be used in some cases to show that directed harmonic currents are unique. This can be viewed as a type of unique ergodicity in this setting, see \cite{fornaess-sibony:finite-energy}, \cite{dinh-sibony:unique-ergodicity} and also \cite{perez-garrandes}.

The proof of Theorem \ref{thm:main-thm} consists of an application of the theory of densities of positive closed currents of Dinh and Sibony, \cite{dinh-sibony:density}. This notion can be used to define the intersection of positive closed currents in a quite general setting and it seems to be a well suited tool to deal with intersections of directed currents (see \cite{dinh-sibony:unique-ergodicity} for a recent application to foliations in $\pr^2$).

The paper is organized as follows. In Section \ref{sec:preliminaries} we set some notation and state the basic facts about laminations and foliated cycles. Section \ref{sec:densities} is devoted to a brief overview of the theory of densities. The main points are the fact that density classes encode the cohomological intersection (see Proposition \ref{prop:cup-product}) and that representatives of these classes can be computed using local coordinates, see Theorem \ref{thm:local-comp}. Therefore, Theorem \ref{thm:main-thm} can be proved using the local expression of a foliated cycle in a flow box, which is done in Section \ref{sec:proof-main-thm}. Finally, in Section \ref{sec:applications} we apply our theorem to some concrete examples.

\subsection*{Acknowledgements} This work was supported by a grant from Région Île-de-France and part of it was done during a visit to the National University of Singapore. I am thankful for their support. I would also like to express my gratitude to  T.--C. Dinh for his valuable advices. 

\section{Preliminaries}  \label{sec:preliminaries}
\subsection{Laminations by complex manifolds} \label{sec:laminations}

Let $\Sigma$ be a compact metric space. A \textit{lamination} $\mathcal L$ of $\Sigma$ by complex manifolds of dimension $q$ is an atlas of homeomorphisms $\varphi_i: U_i \to \mathbb D^q \times K_i$, where $\D^q$ is the unit disc on $\C^q$ and $K_i$ is a topological space such that the coordinate changes $\varphi_{ij} = \varphi_j \circ \varphi_i^{-1}$ are of the form
\begin{equation} \label{eq:foliated-charts}
\varphi_{ij}(z,t) = (f_{ij}(z,t), \gamma_{ij}(t)),
\end{equation}
with $f_{ij}$ holomorphic in the first variable.

We will call such an open set $U_i$ a \textit{flow box} and the sets $\Pi_t = \varphi_i^{-1}(\D^q \times \{t\})$ are called \textit{plaques}. A subset $L \subset \Sigma$ is called a \textit{leaf} of $\mathcal{L}$ if it is a minimal connected set with the property that if $L$ meets a plaque $\Pi$ then $\Pi \subset L$.

Let now $X$ be a compact complex manifold. We say that $\mathcal L$ is an \textit{embedded lamination} on $X$ if $\Sigma \subset X$,  the transversals $K_i$ are contained in $\D^{n-q}$ and the charts $\varphi_i^{-1}$ extend to homeomorphisms between $\D^q \times \D^{n-q}$ and an open set of $X$ which are holomorphic in the first coordinates.

The following lemma shows that, in holomorphic coordinates, the plaques are given by disjoint graphs of holomorphic functions.

\begin{lemma} \label{lemma:lamination-hol-coord}
Let $\mathcal L$ be a lamination of dimension $q$ embedded on $X$. Then for every point $p$ in a flow box $U_i$ there is an open set $V_i \ni p$ and a holomorphic chart $\psi_i: V_i \to \D^q \times \D^{n-q} $ centered at $p$ such that $$\psi_i(\Pi_t) = (z',h_t(z')),\,\, z \in \D^q,$$ where $h_t$ are uniformly bounded holomorphic maps with values in $\C^{n-q}$ depending continuously in $t$ such that $h_t(0) = t$ and $h_0 \equiv 0$.
\end{lemma}

\begin{proof}[Sketch of the proof] Let $p$ be a point in a flow box $U_i$ and let $\varphi_i: U_i \to \D^q \times \D^{n-q}$ be a foliated chart centered at $p$. This is a homeomorphism that restricts to a holomorphic map on each plaque  $\Pi_t = \varphi_i^{-1}(\D^q \times \{t\})$. Choose a \textit{holomorphic} chart $\psi_i: V_i \to \D^q \times \D^{n-q}$ centered at $p$ such that $\psi_i(\Pi_0)$ is given by $z'' = 0$, where $z = (z',z'')$ are the standard coordinates in $ \D^q \times \D^{n-q}$. The projection on the first coordinate $\D^q \times \D^{n-q} \to \D^q$ restricts to a one--to--one holomorphic map on each $\psi_i(\Pi_t)$, so these sets are given by disjoint holomorphic graphs $(z',f_t(z'))$ varying continuously with $t$. We can take $h_0 = f_0 \equiv 0$ and  $h_t(z') = \psi_t(f_t(z'))$ for $t \neq 0$, where $\psi_t$ is a continuous family of linear isomorphims of $\C^{n-q}$ mapping $f_t(0)$ to $t$.
\end{proof}

\begin{remark}
The continuity of the $h_t$ in $t$ in the above lemma is with respect to the $\mathcal C^\infty$ topology. Notice that, since the $h_t$ are holomorphic, this equivalent to the continuity with respect to the $\mathcal C^0$ topology. This can be seen from Cauchy's formula.
\end{remark}

In what follows we will also call an open set $V_i$ as in Lemma \ref{lemma:lamination-hol-coord} a \textit{flow box}.

\subsection{Foliated cycles}

Let $\mathcal L$ be a lamination of dimension $q$ embedded on $X$. A \textit{foliated cycle} directed by $\mathcal L$ is a positive closed current $T$ on $X$ of bi-dimension $(q,q)$ which in a flow box $\varphi_i: U_i \to \mathbb D^q \times K_i$ is given by $$T = \int_{K_i} [{\D^q \times \{t\}}] \, \diff \mu_i(t),$$ where $[\D^q \times \{t\}]$ denotes the current of integration along $\D^q \times \{t\}$ and $\mu_i$ is a positive measure on $K_i$. More precisely, if $\alpha$ is a differential $(q,q)-$form with compact support in $U_i$, then 
\begin{equation*}
\langle T, \alpha \rangle = \int_{K_i} \left(\int_{\D^q \times \{t\}} (\varphi_i)_*\alpha \right ) \diff \mu_i(t).  
\end{equation*}

The fact that the above expressions glue together to give a globally defined current corresponds to the invariance relation $$ \mu_i(E) = \mu_j (\gamma_{ij}(E)) \, \text{ for every Borel set } E \subset K_i,$$ where $\gamma_{ij}: K_i \to K_j$ are as in (\ref{eq:foliated-charts}). This gives a correspondence between foliated cycles and transverse invariant measures. We will say that $T$ is diffuse if in every flow box the corresponding measure $\mu_i$ is diffuse (i.e.\, does not charge points).

\begin{remark}
A current $T$ as above is sometimes called strongly directed or uniformly laminar. There is also a notion of weakly directed current (cf. \cite{fornaess-sibony:laminations}, Definition $6$). These notions coincide for foliations due to a classical result of Sullivan, \cite{sullivan:cycles}. In \cite{fornaess-wang-wold} the authors show that every weakly directed current is a foliated cycle when $n = 2$. They also give a counter-example in dimension $3$.
\end{remark}

\begin{example}
Suppose that $\mathcal L$ has a compact leaf $L$. Then the current of integration along $L$ is a foliated cycle directed by $\mathcal L$. The transverse measures in this case are Dirac masses.
\end{example}

\begin{remark} \label{rmk:atoms-leaves}
Conversely, we can show that if in some flow box the transverse measure of a foliated cycle $T$ contains an atom then $T$ ``contains'' the current of integration along a compact leaf $L$ of $\mathcal L$, that is, $T = T' + c[L]$ where $T'$ is a foliated cycle directed by $\mathcal L$ and $c>0$ is a constant. In particular, if a lamination $\mathcal L$ has no compact leaves then any foliated cycle directed by $\mathcal L$ is diffuse.
\end{remark}

\begin{example}[Ahlfors' currents] \label{ex:ahlfors}
Suppose that $\mathcal L$ is a lamination by Riemann surfaces, that is $q=1$. If $\mathcal L$ admits a parabolic leaf $L$ (i.e., a leaf uniformized by the complex plane) then the uniformizing map $\phi: \C \to L$ gives rise to foliated cycles directed by $\mathcal L$. More precisely, let $a_r$ be the area of $\phi(\D_{r})$ and $\ell_r$ be the length of $\phi(\del \D_{r})$.  If $r_n$ is a sequence tending to infinity such that $\lim_{r_n \to +\infty} \frac{\ell_{r_n}}{a_{r_n}} = 0$, then the normalized currents of integration along $\phi(\D_{r_n})$ converge to a closed positive current directed by $\mathcal L$. Ahlfors showed that such sequences always exist. We can also produce foliated cycles using different notions of parabolicity (see Remark \ref{rmk:parabolicity}).
\end{example}

\section{Density of positive closed currents} \label{sec:densities}
We review in this section the basic facts of the theory of densities of positive closed currents introduced by T.--C.Dinh and N. Sibony. This will be the main tool in the proof of Theorem \ref{thm:main-thm}.  We refer to \cite{dinh-sibony:density} for the complete exposition.

\subsection{Tangent currents}
Let $T$ be a positive closed current of bi--degree $(p,p)$ defined in a neighborhood of the origin in $\C^n$. The Lelong number of $T$ at $0$ is defined by $$\nu(T,0) = \lim_{ r \to 0} \frac{\|T\|_{B(0,r)}}{(2\pi)^{n-p}r^{2n-2p}},$$ where $\|T\| = T \wedge \beta^{n-p}$ is the trace measure with respect to the standard Kähler form $\beta$ on $\C^n$.

The Lelong number can be characterized by the following geometric construction. Let $A_\lambda:\C^n \to \C^n$ be defined by $A_\lambda(z) = \lambda z$ for $\lambda \in \C^*$. Then the family of currents $T_\lambda = (A_\lambda)_* T$, $|\lambda| \geq 1$ is relatively compact and any limit current $S$ is defined in all of $\C^n$. We can view $S$ as a current in $\pr^n$  by considering its trivial extension through the hyperplane at infinity.

\begin{lemma}
The class of $S$ in $H^{2p}(\pr^n,\C)$ is equal to $\nu(T,0)$ times the class of a linear subspace of dimension $n-p$.
\end{lemma}


Now let $T$ be a positive closed $(p,p)-$current defined on a compact Kähler manifold $X$ and let $V \subset X$ be a compact complex submanifold. Denote by $E$ the normal bundle to $V$ in $X$ and $\overline E$ its canonical compactification, defined as $\overline E = \pr(E \oplus \C)$. The multiplication by $\lambda \in \C^*$ along the fibers defines an endomorphism $A_\lambda: E \to E$ that extends trivially to $\overline E$.

Inspired by the above discussion we may want to define a tangent of $T$ along $V$ as a limit of $T$ under the dilations $A_\lambda$. The main problem is that we cannot transport $T$ to a current in $E$ in a holomorphic way. Nevertheless, there are maps $\tau$  from a neighborhood of $V$ in $X$ and a neighborhood of the zero section in $E$ that are diffeomorphisms and whose differential at $V$ is the identity in the normal direction (see \cite{dinh-sibony:density}, Definition 2.14). These maps are called admissible and they suffice to define tangent currents.

\begin{theorem} [Existence of tangent currents, \cite{dinh-sibony:density} Thm 1.1] \label{thm:existence-tangent-currents}
Let $X$,$V$,$T$,$E$,$\overline E$, $A_\lambda$ and $\tau$ be as above. Then the family of currents $T_\lambda = (A_\lambda)_* \tau_* T$ on $E$ is relatively compact. If $S$ is any of its cluster values, viewed as a positive closed current in $E$ extended by zero to $\overline E$, then its cohomology class in $H^{2p}(\overline E,\C)$ is independent of the choice of  $\, \tau$ and $S$.
\end{theorem}

\begin{definition} \label{def:tangent-class}
A current $S$ on $\overline E$ as in Theorem \ref{thm:existence-tangent-currents} is called a \textit{tangent current} to $T$ along $V$. The class $\{S\} \in H^{2p}(\overline E, \C)$, which is independent of $S$, is called the \textbf{total tangent class} of $T$ along $V$ and it is denoted by $\kappa^V(T)$.
\end{definition}

\subsubsection{Local computation of tangent currents}
An important feature of the notion of tangent currents is that they can be computed using local coordinates (even though the current $T$ must be globally defined). This property will be crucial in the proof of Theorem \ref{thm:main-thm}.

Let $V \subset X$ as above and denote $\ell = \dim V$ and $n = \dim X$. Let $V_0$ be a small open set of $V$ and $U_0$ a small neighborhood of $V_0$ in $X$. We may assume that $U_0$ is biholomorphic to $\D^n =  \D^\ell \times \D^{n-\ell}$ via local holomorphic coordinates $x = (x',x'')$ in which $V_0$ is given by $\{x'' = 0\}$. In these coordinates the normal bundle $E$ to $V_0$ is holomorphically identified with the trivial $\C^{n-\ell}-$ bundle over $\D^\ell$ and $V_0$ is  identified with its zero section.

Denote by $\tau: U_0 \to \D^\ell \times \D^{n-\ell}$ the above holomorphic chart. 
Let $A_\lambda (x',x'') = (x',\lambda x'')$, $\lambda \in \C^*$ be the multiplication along the fibers of $E$ and define $$T_\lambda = \tau^* \, (A_\lambda)_* \, \tau_* \, T.$$ The following result follows from \cite{dinh-sibony:density}, Proposition 4.4.

\begin{theorem}[Local computation] \label{thm:local-comp}
The family $\{T_\lambda, |\lambda| \geq 1\}$ is relatively compact. Moreover, if $(\lambda_n)$ is a sequence converging to infinity such that $T_{\lambda_n} \to S$, then $S$ is a positive closed $(p,p)-$current on $E|_{V_0}$ independent of the choice of the local chart $\tau$ as above. 
\end{theorem}



\subsection{Density classes and intersection product} Let $X$ be a compact Kähler manifold of dimension $n$. Given two positive closed currents $T_1$ and $T_2$ of bi--degrees $(p_i,p_i)$, $i=1,2$ on $X$ we will define the density of the pair $(T_1,T_2)$. When $T_2$ is the current of integration along a submanifold $V$ we recover above the notion of tangent current along $V$. 

Let $\mathbf T = T_1 \otimes T_2$ be the product current on $X \times X$. Notice that $\mathbf T$ is of bi-degree $(p,p)$, where $p = p_1 + p_2$. Denote by $\Delta \subset X \times X$ the diagonal. Let $E$ be the normal bundle to $\Delta$ in $X \times X$. It is a rank $n$ vector bundle and, under the natural isomorphism $\Delta \simeq X$, $E$ identifies with the tangent bundle of $X$. 

\begin{definition}
A tangent current to $\mathbf T$ along $\Delta$ is called a density current associated with $T_1,T_2$. The tangent class $\kappa(T_1,T_2) := \kappa^\Delta(\mathbf T)$ is called the \textbf{total density class} of $T_1,T_2$.
\end{definition}


These ideas can be further developed and culminate in the definition of an intersection product of $T_1$ and $T_2$ which is compatible with the cohomological cup product and coincides with classical notions of intersection. The interested reader may consult Section $5$ in \cite{dinh-sibony:density}. We state here only a particular case that we will use.
\begin{proposition} \label{prop:cup-product}
Suppose that $p_1 +p_2 \leq n$. If the total density class $\kappa(T_1,T_2)$ vanishes then $\{T_1\} \smallsmile \{T_2\} = 0$ in $H^{2n - 2p_1 - 2p_2}(X,\C)$.
\end{proposition}

\section{Proof of the main theorem}  \label{sec:proof-main-thm}
This section is devoted to the proof of Theorem \ref{thm:main-thm}. By Theorem \ref{thm:local-comp} we will be able to deduce the global result from a computation using local coordinates, so we begin by considering the local case.

\subsection{Local result} \label{sec:local-problem}
 Let $\mathcal L$ be a Lipschitz lamination of dimension $q$ embedded on a compact Kähler manifold $X$ of dimension $n$.  From Lemma \ref{lemma:lamination-hol-coord} we can choose holomorphic coordinates $x = (x',x'')$ in a flow box $B \simeq \D^q \times \D^{n-q}$ in which the plaques of $\mathcal L$ in $B$ are given by $$ L_a = \{x'' = h_a(x')\},$$ for some uniformly bounded holomorphic functions  $h_a: \D^q \to \D^{n-q}$ depending continuously on $a = h_a(0)$ and $h_0 \equiv 0$.

Let $T$ be a diffuse foliated cycle of dimension $q$ on $X$ directed by $\mathcal L$. In the flow box $B$, $T$ is given by $$ T = \int [L_a] \, \diff \mu(a),$$ where $[L_a]$ denotes the current of integration along $L_a$ and $\mu$ is a diffuse positive measure on $\D^{n-q}$.

Let $\mathcal L \times \mathcal L$ denote the self-product lamination on $X \times X$. The coordinates on the flow box $B$ induce natural holomorphic coordinates $(x,y)$ on $B \times B$ in which the plaques of $\mathcal L \times \mathcal L$ are given by
 $$ L_{a,b} := L_a \times L_b = \{x'' = h_a(x'), \, y'' = h_a(y')\}.$$
 
 The product $ \mathbf T := T \otimes T$ is a foliated cycle of dimension $2q$ on $X \times X$ directed by $\mathcal L \times \mathcal L$ which is given,   on $B \times B$, by
 $$  \mathbf T = \int [L_{a,b}] \, \diff \mathfrak M (a,b), \text{ where } \mathfrak M = \mu \otimes \mu.$$

We will need the following simple result. It can be proven, for instance, using Fubini's Theorem.
\begin{lemma} \label{lemma:mass-diagonal}
Let $\mu$ be a finite measure on a separable Borel space $\mathcal X$. If $\mu$ has no atoms then $\mathfrak M = \mu \otimes \mu$ gives no mass to the diagonal $\Delta \subset \mathcal X \times \mathcal X$.
\end{lemma}

Since we are  interested in computing the tangent currents of $\mathbf T$ along the diagonal $\Delta = \{x=y\} $ it is convenient to work in new holomorphic coordinates given by $$(z,w) = (x, y - x)$$ and new parameters given by $\alpha = (\alpha_1,\alpha_2)$,  where $\alpha_1 = a$ and $\alpha_2 = b-a$.

Write $z = (z',z'')$ with $z' \in \C^q$ and $z'' \in \C^{n-q}$ and similarly $w = (w',w'')$. In this new coordinate system the diagonal is given by $\Delta = \{w = 0\}$ and $L_{a,b}$ transforms to 
$$\Gamma_\alpha = \left \{( z ',f_\alpha( z', w'), w', g_\alpha(z',w')) \right \},$$
where $f_\alpha(z',w') = h_{\alpha_1}(z')$ and  $g_\alpha(z',w') = h_{\alpha_1 + \alpha_2}(z'+w') - h_{\alpha_1}(z'))$. Notice that $f_\alpha(0,w') = \alpha_1$ and $g_\alpha(0,0) = \alpha_2$.

The expression of $\mathbf T$ becomes  $$  \mathbf T = \int [\Gamma_\alpha] \, \diff \mathfrak M (\alpha).$$

We will need the following properties of $f_\alpha$, $g_\alpha$ and $\mathfrak M$:
\begin{enumerate}
\item The parameter $\alpha \in \D^{2n-2q}$ is determined by $\alpha = (f_\alpha(0,0),g_\alpha(0,0))$,
\item $(f_\alpha,g_\alpha)$ depends on $\alpha$ on a bi-Lipschitz way, \label{cond:lip}
\item $f_\alpha$ depends only on $\alpha_1$, \label{cond:indep-alpha2}
\item If $\alpha_2=0$ then $g_\alpha(z',0) \equiv 0$ and if $\alpha_2 \neq 0$ then $g_\alpha(z',0) \neq 0$ for every $z' \in \D^q$. \label{cond:geometric}
\item The measure $\mathfrak M$ gives no mass to the set $\{\alpha_2=0\}$, see Lemma \ref{lemma:mass-diagonal}. 

\end{enumerate}

\begin{remark} \label{rmk:lip-derivatives}
(a) Property (\ref{cond:lip}) means that  for every $0< \rho <1$ there is a constant $c = c(\rho) >0$  such that
\begin{equation} \label{eq:basic-lip}
c^{-1} \|\alpha- \beta \| \leq \max \{\|f_\alpha(z',w') - f_\beta(z',w')\|, \|g_\alpha(z',w') - g_\beta(z',w')\| \} \leq c \, \|\alpha - \beta\|,
\end{equation}
 for every $(z',w')$ such that $\|z'\|,\|w'\| \leq \rho$.
 
(b) From Cauchy formula, we have analogous upper bounds for the derivatives of $f_\alpha$ and $g_\alpha$ of any order.

(c) Property (\ref{cond:geometric}) corresponds to the geometric fact that the plaques $L_{a,b}$ with $a\neq b$ don't intersect the diagonal and the plaques $L_{a,a}$ intersect the diagonal along a $q$-dimensional manifold.

\end{remark}

Let $A_\lambda(z,w) = (z,\lambda w)$ be the dilation in the direction normal to $\Delta$ and
set $$\mathbf T_\lambda = (A_\lambda)_* \mathbf T, \,\, \text{ for } \,\,|\lambda| \geq 1.$$

The main result in this section is the following.

\begin{proposition} \label{prop:local-result}
Let $\mathcal L$, $\Gamma_\alpha$, $\mathbf T$ and $\mathbf T_\lambda$ be as above. Then $$\lim_{|\lambda| \to \infty} \mathbf T_\lambda = 0$$ in a neighborhood of the origin.
\end{proposition}

The proof will follow from the following basic estimates.

 \begin{lemma} \label{lemma:1}
Fix $0< \rho <1$. There are  constants $c_1,c_2,c_3 > 0$ such that the following estimates hold for every  $(z',w')$ such that $\|z'\|,\|w'\| \leq \rho$
\begin{equation} \label{eq:main-ineq}
c_1 \cdot (\|\alpha_2\| - c_2 \, \|\alpha\| \, \|w'\|) \leq \| g_\alpha(z',w') \| \leq c_3 \cdot (\|\alpha_2 \| + \|\alpha\| \, \|w'\|). 
\end{equation}
 \end{lemma}

\begin{proof}

From Property  (\ref{cond:geometric}) and the  Lipschitz condition (\ref{eq:basic-lip}) we get $$\|g_\alpha(z',0)\| = \|g_{\alpha_1,\alpha_2}(z',0) \| = \|g_{\alpha_1,\alpha_2}(z',0) - g_{\alpha_1,0}(z',0)\| \leq c \|\alpha_2\|.$$

From Remark \ref{rmk:lip-derivatives}--(b) we have that $\|Dg_\alpha (z',w')\| \leq k \|\alpha\|$ for $\|z'\|,\|w'\| \leq \rho$ and some constant $k>0$, so by the mean value theorem we get $$\|g_\alpha (z',w') - g_\alpha (z',0)\| \leq k \, \|\alpha\| \, \|w'\|.$$

Combining these two inequalities gives  $$\| g_\alpha(z',w') \| \leq  c_3(\|\alpha_2 \| + \|\alpha\| \, \|w'\|)$$ for some constant $c_3 > 0$.

To prove the left inequality in (\ref{eq:main-ineq}) we use Property (\ref{cond:indep-alpha2}) and the estimate (\ref{eq:basic-lip}). We get
\begin{equation*}
\begin{split}
\|g_\alpha(z',0)\| &= \|g_{\alpha_1,\alpha_2}(z',0) \| = \|g_{\alpha_1,\alpha_2}(z',0) - g_{\alpha_1,0}(z',0)\| \\
&= \max \left\{ \|g_{\alpha_1,\alpha_2}(z',0) - g_{\alpha_1,0}(z',0)\|, \|f_{\alpha_1,\alpha_2}(z',0) - f_{\alpha_1,0}(z',0)\|\right\} \\
&\geq c^{-1} \|\alpha_2\|.
\end{split}
\end{equation*}

Using again the fact that  $\|Dg_\alpha (z',w')\| \leq k \|\alpha\|$ for $\|z'\|,\|w'\| \leq \rho$ and the mean value theorem we get $$\| g_\alpha(z',w') \| \geq c_1 \cdot (\|\alpha_2 \| - c_2 \, \|\alpha\| \, \|w'\|),$$ for some constants $c_1,c_2 > 0$. This completes the proof of the Lemma.
\end{proof}

We will also need the following standard result, saying that the volume of a bounded family of graphs is uniformly bounded. We first recall some basic facts.

Let $M$ be a complex manifold equipped with a hermitian metric $\omega$. If $Y$ is a complex submanifold of dimension $m$ of $M$ then the metric $\omega$ induces a metric on $Y$. By Wirtinger's Theorem the corresponding volume form on $Y$ is $\frac{1}{m!} \omega^m$. This means that the $2m$-dimensonal volume of $Y$ over a Borel set $K$, denoted by $\vol_{2m}(A \cap K)$, is given by $\frac{1}{m!} \int_{A\cap K} \omega^m$. This quantity is equal to $\frac{1}{m!}$ times the mass of the trace measure $[A] \wedge \omega^m$ over $K$. 

\begin{lemma} \label{lemma:cauchy} Let $\Omega$ be a domain in $\C^m$, $\mathcal F$ be a family of holomorphic maps $f:\Omega \to \C^N$ and denote by $\Gamma_f \subset \C^{m+N}$ the graph of $f$. Suppose that there exists an $M > 0$ such that $\|f(z)\| \leq M$ for every $z \in \Omega$ and every $f \in \mathcal F$. Then for every compact $K \subset \Omega$ there is a constant $C>0$ such that $$ \vol_{2m}(\Gamma_f \cap (K \times \C^N)) \leq C\, \, \text{ for every } \, f \in \mathcal F,$$ where the volume is computed with respect to the standard Hermitian metric on $\C^{m+N}$.
\end{lemma}

\begin{proof}
Let $\omega = \sum_{i=1}^{m+N} \diff z_i \wedge \diff \overline z_i$ be standard Hermitian form on $\C^{m+N}$. For each $f \in \mathcal F$ consider the parametrization $\phi_f: \Omega \to \Gamma_f$, $\phi_f(x) = (x,f(x))$. From Wirtinger's Theorem we get that $$ \vol_{2m}(\Gamma_f \cap (K \times \C^N)) =  \frac{1}{m!} \int_{\Gamma_f \cap (K \times \C^N)}  \omega^m = \frac{1}{m!} \int_K \phi_f^* \omega^m.$$ Since the functions $f$ are uniformly bounded it follows from Cauchy's estimates that the first derivatives of $f$ are uniformly bounded over $K$ . This implies that the coefficients of the $(m,m)-$form $\phi_f^* \omega^m$ are bounded in $K$, uniformly in $f$. This gives the result.
\end{proof}

To prove Proposition \ref{prop:local-result} we will look at the action of the dilation $A_\lambda$ on each plaque $[\Gamma_\alpha]$  and consider the complementary cases when $\|\alpha_2 \| \leq \frac{1}{|\lambda|}$ and when $\|\alpha_2 \| > \frac{1}{|\lambda|}$.

 \begin{lemma} \label{lemma:3}
 The family of currents $\mathcal F = \{(A_\lambda)_*[\Gamma_\alpha] \, : \, |\lambda| \geq 1, \|\alpha_2 \| \leq \frac{1}{|\lambda|} \}$ is relatively compact. In other words the mass of  $(A_\lambda)_*[\Gamma_\alpha]$ on a compact set is bounded uniformly in $(\lambda,\alpha)$ when $\|\alpha_2 \| \leq \frac{1}{|\lambda|}$.
 \end{lemma}

\begin{proof}
We will consider only the pairs $(\lambda,\alpha)$ satisfying $\|\alpha_2 \| \leq |\lambda|^{-1}$. Recall that the para--meters $\alpha$ vary on a bounded set.  The support $\Gamma_\alpha^\lambda$ of $(A_\lambda)_*[\Gamma_\alpha]$ is a graph over the $(z',w')$ coordinates given by
\begin{equation} \label{eq:gamma-alpha-lambda}
\Gamma_\alpha^\lambda =  \left\{ ( z ',f_\alpha( z', \lambda^{-1}w'), w', \lambda g_\alpha(z',\lambda^{-1} w')) \right\}, \;\;\; \|z'\| < 1,\,\, \|w'\| < |\lambda|
\end{equation}
and the mass of $(A_\lambda)_*[\Gamma_\alpha]$ over a compact set $K$ is $(2q)!$ times the $4q$-dimensional volume of $\Gamma_\alpha^\lambda \cap K$.

We claim that, for a fixed $0<\rho<1$, the above graphs are uniformly bounded over the set $\|z'\|, \|w'\|  \leq \rho$. Indeed, the functions $f_\alpha( z', \lambda^{-1}w')$  are bounded uniformly in $\alpha$ and $\lambda$, and Lemma \ref{lemma:1} gives  $$\|\lambda g_\alpha(z',\lambda^{-1} w')\| \leq  |\lambda|\,c_3 \cdot (\|\alpha_2 \| + \|\alpha\| \, |\lambda|^{-1}\|w'\|) = c_3 \cdot (|\lambda| \, \|\alpha_2 \|  +  \|\alpha\| \, \|w'\|) \leq c'_3(1+ \rho)$$ for $\|z'\|,\|w'\| \leq \rho$, since we are assuming $\|\alpha_2 \| \leq |\lambda|^{-1}$.

It follows from Lemma \ref{lemma:cauchy} that the $2q$-dimensional volumes of the $\Gamma_\alpha^\lambda$ are locally boun--ded, uniformly in $(\lambda,\alpha)$ when $\|\alpha_2 \| \leq \frac{1}{|\lambda|}$.
\end{proof}

\begin{lemma} \label{lemma:4}
There exists a compact neighborhood of the origin $K$ such that $(A_\lambda)_*[\Gamma_\alpha]$ has no mass on $K$ for every pair $(\lambda,\alpha)$ such that $\|\alpha_2 \| > \frac{1}{|\lambda|}$. \end{lemma}

\begin{proof}
The left inequality in (\ref{eq:main-ineq}) together with the assumption that $\|\alpha_2 \| > |\lambda|^{-1}$ gives $$\|\lambda g_\alpha(z',\lambda^{-1} w')\| \geq c_1 \, |\lambda| \, \cdot (\|\alpha_2 \| - c_2 \, |\lambda|^{-1} \, \|\alpha\|\, \|w'\|) \geq c_1 - k \|w'\|,$$ for some constant $k>0$.

Hence, if $\| w' \| \leq \frac{c_1}{3k}$ then $\|\lambda g_\alpha(z',\lambda^{-1} w')\| > \frac{c_1}{2}$, meaning that  the support $\Gamma_\alpha^\lambda$ of $(A_\lambda)_*[\Gamma_\alpha]$  (see \ (\ref{eq:gamma-alpha-lambda})) lies outside the compact set $$K= \left \{\| z' \| \leq \frac{1}{2}, \| w' \| \leq \frac{c_1}{3k}, \| z'' \| \leq \frac{1}{2}, \| w'' \| \leq \frac{c_1}{2} \right \}.$$ The lemma follows. 
\end{proof}

\begin{proof}[Proof of Proposition \ref{prop:local-result}]
Dividing the integral defining $\mathbf T_\lambda$ in two parts we get
 \begin{equation*}
    \begin{split}
      \mathbf T_\lambda = \int (A_\lambda)_*[\Gamma_\alpha] \, \diff \mathfrak M (\alpha)  &= \int_{\left \{ \|\alpha_2 \| \leq \frac{1}{|\lambda|} \right \}} (A_\lambda)_*[\Gamma_\alpha] \, \diff \mathfrak M (\alpha) + \int_{\left \{  \|\alpha_2 \| > \frac{1}{|\lambda|} \right \}} (A_\lambda)_*[\Gamma_\alpha] \, \diff \mathfrak M (\alpha) \\
      &= \mathbf T'_\lambda + \mathbf T''_{\lambda}.
    \end{split}
 \end{equation*}
 
Take a compact neighborhood of the origin $K$ as in Lemma \ref{lemma:4}, so that $\mathbf T''_\lambda$ is zero on $K$ for every $\lambda$. It follows from Lemma \ref{lemma:3} that there is a constant $M > 0$, independent of $\lambda$, such that the mass of $\mathbf T'_\lambda$ over $K$ is bounded by $M$ times $\mathfrak M(\{\|\alpha_2 \| < |\lambda|^{-1}\})$. The fact that $\mathfrak M$ gives no mass to the set $\{\alpha_2 = 0\}$ shows that $\mathbf T'_\lambda \to 0$ on $K$ as $|\lambda| \to \infty$. The proposition follows.
\end{proof}

\subsection{Global result} \label{sec:global-result}

The local computations of the last section allow us to prove our main theorem. 

\begin{proof}[Proof of Theorem \ref{thm:main-thm}] Notice first that if $q < \frac{n}{2}$ then $2n-2q > n$, so $H^{2n-2q,2n-2q}(X) = 0$ for dimensional reasons. Therefore the product $\{T\} \smallsmile \{T\}$ is trivially zero in this case. We may assume then that $q \geq  \frac{n}{2}$. According to Proposition \ref{prop:cup-product} it suffices to show that the total density class $\kappa(T,T)$ vanishes.

Let $\mathbf S$ be a tangent current to $\mathbf T$ along $\Delta$. If $U$ is a small open set  intersecting the support of $\mathbf T$  and contained in a flow box $B \times B \subset X \times X$ we can choose coordinates $(z,w)$ in $U$ as in Section \ref{sec:local-problem}, so that $\Delta$ is given by $\{w=0\}$. From Theorem \ref{thm:local-comp} we can compute the restriction of $\mathbf S$ to $U$ as the limit $$\mathbf S = \lim_{\lambda_n \to \infty} \mathbf T_{\lambda_n},$$ where $A_\lambda (z,w) = (z, \lambda w)$, $ \mathbf T_\lambda = (A_\lambda)_* \mathbf T$ and $\lambda_n \to \infty$.

From Proposition \ref{prop:local-result} we have that $\lim_{|\lambda| \to \infty} \mathbf T_\lambda = 0$, so any tangent current to $\mathbf T$ along $\Delta$ vanishes on $U$. Since $U$ is arbitrary we conclude that the zero current is the only tangent current to $\mathbf T$ along $\Delta$. Therefore the tangent class $\kappa^\Delta(\mathbf T) = \kappa(T,T)$ is zero,  concluding the proof.
\end{proof}

\begin{remark} \label{rmk:T^T}
For $q \geq \frac{n}{2}$, the vanishing of the class $\kappa(T,T)$ implies that $T \curlywedge T = 0$, where the wedge product is in the sense of Dinh--Sibony (see \cite{dinh-sibony:density}, Section 5).
\end{remark}

\section{Applications} \label{sec:applications}

We give in this section some consequences of Theorem \ref{thm:main-thm}.

\subsection{Foliated cycles on projective spaces and other manifolds}

Recall that for a compact Kähler manifold $X$, the Hodge Decomposition Theorem says that the de Rham cohomology groups of $X$ decompose as $H^k(X,\C) = \bigoplus_{p+q=k} H^{p,q}(X)$, where $H^{p,q}(X)$ is the set of classes that can be represented by a closed $(p,q)-$form (see \cite{voisin:book1}). The numbers $h^{p,q}(X) = \dim H^{p,q}(X)$ are called the \textit{Hodge numbers} of $X$. For $1 \leq p \leq \dim X$ the space $H^{p,p}(X)$ is never zero  since it contains the class of the $p-$th power of a Kähler form.

\begin{theorem}
Let $X$ be a compact Kähler manifold of dimension $n$ and let $\mathcal L$ be an embedded Lipschitz lamination on $X$ of dimension $q \geq n/2$. Suppose that  $h^{n-q,n-q}(X) = 1$. Then there is no diffuse foliated cycle directed by $\mathcal L$.
\end{theorem}

\begin{proof}
Fix a Kähler form $\omega$ in $X$. Let $T$ be a diffused foliated cycle of dimension $q$ directed by $\mathcal L$. From Theorem \ref{thm:main-thm} we have that $\{T\}  \smallsmile \{T\} = 0$. On the other hand, since $h^{n-q,n-q}(X) = 1$ the class $\{T\} \in H^{n-q,n-q}(X)$ equals to  $c\{\omega^{n-q}\}$ for some $c>0$. In particular $\{T\}  \smallsmile \{T\} = c^2 \{\omega^{2n-2q}\} \neq 0$, a contradiction.
\end{proof}

\begin{corollary}
Let $X$ be a compact Kähler manifold with $b_2(X) = 1$. Then there is no diffuse foliated cycle on $X$ directed by a Lipschitz lamination of codimension $1$.
\end{corollary}

From the fact that $h^{p,p}(\pr^n) = 1$ for every $p$ we get 
 
\begin{corollary} \label{cor:foliated-cycles-Pn}
Let $\mathcal L$ be an embedded Lipschitz lamination of dimension $q \geq n/2$ on $\pr^n$. Then there is no diffuse foliated cycle directed by $\mathcal L$.
In particular, if $\mathcal L$ has no compact leaf then there is no foliated cycle directed by $\mathcal L$.
\end{corollary}

The above result is proven in \cite{fornaess-sibony:finite-energy} for $n=2$ and $q=1$. In their proof, a crucial ingredient is the fact that $\pr^2$ is homogeneous under the action of its automorphism group. This is used in order to regularize the foliated cycle. Our approach allows us to bypass this difficulty.

In what follows a Riemann surface is called parabolic if it is uniformized by the complex plane and hyperbolic if  it is uniformized by the the unit disc. The following corollary is consequence of the Ahflors' construction (Example \ref{ex:ahlfors}).

\begin{corollary} \label{cor:h^11=1}
Let $X$ be a Kähler surface with $h^{1,1}(X) = 1$. Let $\mathcal L$ be an embedded Lipschitz lamination by Riemann surfaces on $X$ without compact leaves. Then all the leaves of $\mathcal L$ are hyperbolic.
\end{corollary}

\begin{corollary}
Let $X$ be a Kähler surface with $h^{1,1}(X) = 1$. Let  $\mathcal F$  be a (singular) holomorphic foliation of $X$  without compact leaves. If $\mathcal F$ admits a parabolic leaf $L$ then $\overline{L}$ intersects the singular set of $\mathcal F$.
\end{corollary}

\begin{proof}
Suppose that there is a parabolic leaf $L$ whose closure does not meet the singular set of $\mathcal F$. In this case, the foliation $\mathcal F$ restricts to a non-singular lamination $\mathcal L$ in a neighborhood of $\overline L$ and the Ahlfors' construction (Example \ref{ex:ahlfors}) gives a foliated cycle directed by $\mathcal L$, which is impossible in view of Corollary \ref{cor:h^11=1}.
\end{proof}

The above results hold in particular when $X$ is the complex projective plane. In this case, the condition of having no compact leaves is generic, see \cite{jouanolou:pfaff}.

\begin{remark}  \label{rmk:parabolicity} Some authors consider different notions of parabolicity and use them to produce $\diff-$closed and $\ddc$--closed currents in the same spirit as Ahlfors (see \cite{burns-sibony}, \cite{dethelin:ahlfors}  and \cite{paun-sibony}). We can also apply the above results to these types of parabolic leaves.
\end{remark}

\subsection{Foliated cycles on complex tori}

Let $X$ be a complex torus. A lamination on $X$ is called linear if it is given by a union of leaves of a linear foliation on $X$.

\begin{remark}
Let $\mathcal L$ be a lamination of dimension $q$ on a complex manifold $X$ that extends to a holomorphic foliation $\mathcal F$ on $X$. Then the foliated cycles directed by $\mathcal L$ are precisely the positive closed $(n-q,n-q)-$currents $T$ whose support is contained in the support of $\mathcal L$ and satisfy $T \wedge \gamma_i = 0$, where $\gamma_i$ are local holomorphic $(1,0)-$forms whose kernel defines $\mathcal F$.
\end{remark}

\begin{proposition}
Let $X$ be a complex torus and let $\mathcal L$  be a Lipschitz lamination of codimension one embedded on $X$. Then there is a foliated cycle directed by $\mathcal L$ if and only if $\mathcal L$ is linear.
\end{proposition}

\begin{proof}
Suppose that $\mathcal L$ is linear. In particular, every leaf of $\mathcal L$ is the image of an affine hyperplane $H \simeq \C^{n-1}$ by the canonical projection $\pi: \C^n \to X$.  A result of de Thélin \cite{dethelin:ahlfors} allows us to construct a positive closed current $T$  of bi-dimension $(n-1,n-1)$ as a limit of normalized currents of integration along $\pi(\D_{r_n})$, where $\D_r$ is an $(n-1)-$dimensional disc of radius $r$ in $H$ and $r_n$ is a sequence tending to infinity. By construction, the current $T$ is a foliated cycle directed by $\mathcal L$.

Let us prove the converse. Recall that, using Hodge theory, we can see that the cohomology ring $H^\bullet(X,\C)$ is isomorphic to the exterior algebra $\bigwedge^\bullet (\C^n)^*$. Under this identification, a positive $(1,1)-$class $\mathbf c$ has vanishing square if and only if it has the from $\mathbf c = i \gamma \wedge \overline \gamma$ for some linear form $\gamma$ of type $(1,0)$.

Let $\mathcal L$ be a  Lipschitz lamination of codimension one on $X$ and let $T$ be a foliated cycle directed by $\mathcal L$. Denote by $\mathbf c$ the class of $T$. By Theorem \ref{thm:main-thm} we have $\mathbf c ^2 = 0$, so by the above remark $\mathbf c = i \gamma \wedge \overline \gamma$ for some $(1,0)$-vector $\gamma$. We claim that $T$ is directed by the linear foliation defined by $\gamma$. Indeed we have $$ 0 = \{T\}^2 = \{T\} \smallsmile i \gamma \wedge \overline \gamma = \{T \wedge i \gamma \wedge \overline \gamma\},$$ and since $T \wedge i \gamma \wedge \overline \gamma$ is a positive current, it must be zero. It is not hard to show from this fact that $T \wedge \gamma = 0$, so $T$ is directed by the linear foliation defined by $ \ker \gamma$.
\end{proof}

\subsection*{Foliated cycles on Hirzebruch surfaces}

We now apply Theorem \ref{thm:main-thm} to describe the foliation cycles on a Hirzebruch surface. Let us first recall some facts about currents on projective bundles. We refer to \cite{dinh-sibony:density}, Section $3$ for more details.

Let $V$ be a Kähler manifold of dimension $\ell$ with Kähler form $\omega_V$ and let $E$ be a rank $r$ holomorphic vector bundle over $V$. The projectivization of $E$, denoted by $\pr(E)$, is the $\pr^{r-1}-$bundle over $V$ whose fiber over $x \in V$ is the projectivization of $E_x$. 

\begin{definition}
Let $S$ be a non-zero positive $(p,p)-$current on $\pr(E)$. The \textit{horizontal dimension} (or $h-$dimension for short) of $S$ is the largest integer $j$ such that $S \wedge \omega_V^j \neq 0$. When this dimension is $0$ we say that $S$ is vertical.
\end{definition}

Given a cohomology class $\mathbf c \in H^{2p}(\pr(E),\C)$ we can apply the Leray-Hisrch Theorem and get a decomposition
\begin{equation*}
\mathbf c = \sum_{j = \max(0,\ell-p)}^{\min(\ell,\ell+r-1-p)} \pi^*(\kappa_j(\mathbf c)) \smallsmile h_{\pr(E)}^{j-\ell+p},
\end{equation*}
where $\kappa_j(\mathbf c)$ is a class in $H^{2 \ell - 2j} (V,\C)$ and $h_{\pr(E)} \in H^2(\pr(E),\C)$ is the tautological class of the associated $\mathcal O_{\pr(E)}(1)-$bundle over $\pr(E)$.

When $\pr(E) = \pr (N_V \oplus \C) =  \overline{N_V}$ is the compactification of the normal bundle to $V$ and $\mathbf c  = \kappa^V(T)$ is the total tangent class to $T$ along $V$ (see Definition \ref{def:tangent-class}) the above classes are denoted by $\kappa_j^V(T)$.

Recall that a cohomology class is pseudo-effective if it contains a positive closed current.

\begin{lemma}[\cite{dinh-sibony:density}, Lemma 3.8]
Let $S$ be a $(p,p)-$current on $\pr(E)$. If $s$ denotes the $h-$ dimension of $S$ then $\kappa_s(\{S\}) \in H^{2 \ell - 2s} (V,\C)$ is pseudo-effective.
\end{lemma}

\begin{lemma} \label{lemma:psef}
Let $X$ be a compact Kähler surface, $V \subset X$ be a smooth curve and $T$ be a $(1,1)-$ current without mass on $V$. If $S$ is a tangent current to $T$ along $V$ then the h-dimension of $S$ in $\overline{N_V}$ is zero. In particular $\kappa_0^V(T)$ is pseudo-effective.
\end{lemma}

\begin{proof}
We have to show that $S \wedge \pi^*\omega_V = 0$, where $\omega_V$ is a Kähler form in $V$. We can work locally. Choose coordinates $(z,w)$ with values in the unit bi-disc in which $V$ is given by $\{w=0\}$  and $N_V \to V$ identifies with the trivial bundle. We can take $\omega_V =  i \diff z \wedge \diff \overline z$.

In these coordinates $S = \lim_{\lambda_n \to \infty} T_{\lambda_n}$, where $T_\lambda = (A_\lambda)_* T$ and $A_\lambda(z,w) = (z,\lambda w)$. Let $0<\rho <1$. Since  $i \diff z \wedge \diff \overline z$ is invariant under $A_\lambda$, the mass of $T_\lambda \wedge  i \diff z \wedge \diff \overline z$ over $\D(\rho)$ is given by $$ \int_{\D(\rho)}  T_\lambda \wedge  i \diff z \wedge \diff \overline z = \int_{\D(\rho)} (A_\lambda)_*  (T \wedge  i \diff z \wedge \diff \overline z) = \int_{\D(\rho, \lambda^{-1}\rho)} T \wedge  i \diff z \wedge \diff \overline z \leq \|T\|_{\D(\rho, \lambda^{-1}\rho)}.$$

Making $\lambda \to \infty$ and using the fact that $T$ has no mass on $V$ gives $\lim_{|\lambda| \to \infty}T_\lambda \wedge i \diff z \wedge \diff \overline z = 0$. By semicontinuity  $S \wedge i \diff z \wedge \diff \overline z = 0$, proving the first part of the Lemma. The last statement follows from Lemma \ref{lemma:psef}.
\end{proof}

\begin{corollary} \label{cor:nef}
Let $X$, $T$ and $V$ be as above. Then $\{T\} \smallsmile \{V\} \geq 0$.
\end{corollary}

\begin{proof}
It follows from the results in Section 5 of \cite{dinh-sibony:density} that the class $\{T\} \smallsmile \{V\} \in H^4(X,\C)$ is the canonical image of $\kappa^V_0(T) \in H^2(V,\C)$ in $H^4(X,\C)$. So by the above lemma $\{T\} \smallsmile \{V\}$ is a pseudo-effective class in $H^4(X,\C) \simeq \C$, hence it is given by a positive real number.
\end{proof}

Recall the the Hirzebruch surface $\Sigma_n$, $n \geq 0$ is defined by $$\Sigma_n = \pr(\mathcal O_{\pr^1} \oplus \mathcal O_{\pr^1}(-n) ),$$ where $\mathcal O_{\pr^1}$ denotes the trivial line bundle over $\pr^1$ and $ \mathcal O_{\pr^1}(-n) $ is the $n^{th}$ exterior power of the tautological line bundle on $\pr^1$.

Every compact complex surface that is a $\pr^1-$bundle over $\pr^1$ is isomorphic to $\Sigma_n$ for some $n \geq 0$. The surface $\Sigma_0$ is equal to $\pr^1 \times \pr^1$ and $\Sigma_1$ is isomorphic to the the projective plane blown-up at a point.

\begin{proposition}
Let $\mathcal L$ be a Lipschitz lamination by Riemann surfaces embedded in $\Sigma_n$, $n \neq 0$. Then every diffuse foliated cycle $T$  directed by $\mathcal L$ is given by $T = \pi^* \nu$ for some diffuse measure on $\pr^1$. In particular, the lamination $\mathcal L$ is a union of fibers of $\pi: \Sigma_n \to \pr^1$.
\end{proposition}

\begin{proof}
The group $H^2(\Sigma_n, \C)$ is generated by $F$ and $C$ , where $F$ is the class of a fiber and $C$ is the class of the section at infinity. Their intersections are given by $$ F^2 = 0, \,\,\,\, F \smallsmile C = 1 \,\,\, \text{ and } \, \, C^2 = -n. $$
If $\nu$ is any probabilty measure on $\pr^1$ then the class of the current $\pi^* \nu$ is equal to $F$. In particular, if $\omega_{FS}$ denotes the Fubini-Study form on $\pr^1$ then $\{\pi^* \omega_{FS}\} = F$.

Let $T$ be a diffuse foliated cycle directed by $\mathcal L$ and denote by $\mathbf c$ its class on $H^2(\Sigma_n, \C)$.  Write $\mathbf c = a F + b C$ for some $a,b \in \R$. Notice that, from Corollary \ref{cor:nef}, we have  $b = C \smallsmile F \geq 0$.  From Theorem \ref{thm:main-thm} we have $$ 0 = \mathbf c^2 = 2ab - b^2\,n = b\,(2a - bn).$$

Assume that $b \neq 0$, then $2a = bn$ and $$\mathbf c \smallsmile C = \left(\frac{bn}{2}F + bC \right) \smallsmile C = \frac{bn}{2} - bn = - \frac{bn}{2} < 0,$$ which contradicts Corollary \ref{cor:nef}. Therefore we must have $b = 0$ and $\mathbf c = a F$.

We claim that $T \wedge \pi^* \omega_{FS} = 0$. Indeed, as noted above the class of $\pi^* \omega_{FS}$ equals to $F$, so the class the positive current $T \wedge \pi^* \omega_{FS}$ is $\mathbf c \smallsmile F = a \, F^2 =0$. The claim follows from the fact that the class of a positive closed current $S$ is zero if and only if $S=0$.

The fact that $T \wedge \pi^* \omega_{FS} = 0$ means that the current $T$ is vertical, i.e. its $h-$dimension is zero. It follows that $T$ is given by an average of currents of integration along the fibers of $\pi$ (see \cite{dinh-sibony:density}, Lemma 3.3 ), that is, $T = \pi^* \nu$ for some measure on $\pr^1$.
\end{proof}

The above result shows in particular that a Lipschitz lamination on $\Sigma_n$ without compact leaves carries no directed  positive closed current. It is worth mentioning that every non-singular holomorphic foliation globally defined on $\Sigma_n$ is given by the canonical fibration (see \cite{brunella:birational}, p. $37$).

For $n=0$ we have that $\Sigma_0 = \pr^1 \times \pr^1$. Arguing as above we can show that the only diffuse foliated cycles are given by $T = \pi_1^* \nu$ or $T =\pi^*_2 \nu$ for some diffuse positive measure $\nu$ on $\pr^1$, where $\pi_i$ denote the projection on each factor. A similar result was obtained by Perez-Garrandes in  \cite{perez-garrandes}, where he also showed that for a Lipschitz lamination in $\pr^1~\times~\pr^1$ the absence of foliated cycles implies the existence of  a unique directed harmonic current of mass one.

\bibliography{refs}
\bibliographystyle{plain}
\end{document}